\documentclass[12pt]{amsart}
\usepackage{amsmath,amssymb,amsbsy,amsfonts,latexsym,amsopn,amstext,cite,
                                               amsxtra,euscript,amscd,bm}
\usepackage{url}

\usepackage{mathrsfs}
\usepackage{enumitem}
\usepackage{color}
\usepackage[colorlinks,linkcolor=blue,anchorcolor=blue,citecolor=blue,backref=page]{hyperref}
\usepackage{color}
\usepackage{graphics,epsfig}
\usepackage{graphicx}
\usepackage{float}
\usepackage{epstopdf}
\hypersetup{breaklinks=true}

\usepackage[np]{numprint}
\npdecimalsign{\ensuremath{.}}

\usepackage{bibentry}

\usepackage[english]{babel}
\usepackage{mathtools}
\usepackage{todonotes}
\usepackage{url}
\usepackage[colorlinks,linkcolor=blue,anchorcolor=blue,citecolor=blue,backref=page]{hyperref}

\usepackage[norefs,nocites]{refcheck}

\DeclareMathOperator*\sign{sign}

\begin{document}

\newtheorem{theorem}{Theorem}
\newtheorem{lemma}[theorem]{Lemma}
\newtheorem{claim}[theorem]{Claim}
\newtheorem{cor}[theorem]{Corollary}
\newtheorem{conj}[theorem]{Conjecture}
\newtheorem{prop}[theorem]{Proposition}
\newtheorem{definition}[theorem]{Definition}
\newtheorem{question}[theorem]{Question}
\newtheorem{example}[theorem]{Example}
\newcommand{\hh}{{{\mathrm h}}}
\newtheorem{remark}[theorem]{Remark}

\numberwithin{equation}{section}
\numberwithin{theorem}{section}
\numberwithin{table}{section}
\numberwithin{figure}{section}

\def\sssum{\mathop{\sum\!\sum\!\sum}}
\def\ssum{\mathop{\sum\ldots \sum}}
\def\iint{\mathop{\int\ldots \int}}

\newcommand{\diam}{\operatorname{diam}}

\def\squareforqed{\hbox{\rlap{$\sqcap$}$\sqcup$}}
\def\qed{\ifmmode\squareforqed\else{\unskip\nobreak\hfil
\penalty50\hskip1em \nobreak\hfil\squareforqed
\parfillskip=0pt\finalhyphendemerits=0\endgraf}\fi}%%

%  use the AMS-Euler Fraktur fonts
%%%%%%%%%%%%%%%%%%%%%%%%%%%%%%%%%%
\newfont{\teneufm}{eufm10}
\newfont{\seveneufm}{eufm7}
\newfont{\fiveeufm}{eufm5}
%%%%%%%%%%%%%%%%%%%%%%%%%%%%%%%%%
%
%  allow automatic size selection in math mode
%
%%%%%%%%%%%%%%%%%%%%%%%%%%%%%%%%%
\newfam\eufmfam
     \textfont\eufmfam=\teneufm
\scriptfont\eufmfam=\seveneufm
     \scriptscriptfont\eufmfam=\fiveeufm
%%%%%%%%%%%%%%%%%%%%%%%%%%%%%%%%%
%
%  \frak works on a single symbol at a time...
%
\def\frak#1{{\fam\eufmfam\relax#1}}
\def\muS{\mu_{\mathsf {S}}}
\def\muM{\mu_{\mathsf {M}}}
\def\muo{\mu_{\bomega}}

\def\muMd{\mu_{\mathsf {M, \delta}}}

\newcommand{\bflambda}{{\boldsymbol{\lambda}}}
\newcommand{\bfmu}{{\boldsymbol{\mu}}}
\newcommand{\bfxi}{{\boldsymbol{\eta}}}
\newcommand{\bfeta}{{\boldsymbol{\eta}}}
\newcommand{\bfrho}{{\boldsymbol{\rho}}}

\def\eps{\varepsilon}

\def\fI{\mathfrak I}
\def\fK{\mathfrak K}
\def\fT{\mathfrak{T}}
\def\fL{\mathfrak L}
\def\fR{\mathfrak R}

\def\fA{{\mathfrak A}}
\def\fB{{\mathfrak B}}
\def\fC{{\mathfrak C}}
\def\fM{{\mathfrak M}}
\def\fS{{\mathfrak  S}}
\def\fU{{\mathfrak U}}

 \def\sfH {\mathsf {H}}
  \def\sfV {\mathsf {V}}
\def\T {\mathsf {T}}
\def\Tor{\mathsf{T}_d}
\def\Tore{\widetilde{\mathrm{T}}_{d} }

\def\sM {\mathsf {M}}

\def\ss{\mathsf {s}}

\def\Kdn{\cK_d(n)}
\def\Kmn{\cK_m(n)}
\def\Kpn{\cK_p(n)}
\def\Kma{\cK_m(a)}

\def \balpha{\bm{\alpha}}
\def \bbeta{\bm{\beta}}
\def \bgamma{\bm{\gamma}}
\def \bdelta{\bm{\delta}}
\def \bzeta{\bm{\zeta}}
\def \blambda{\bm{\lambda}}
\def \bchi{\bm{\chi}}
\def \bphi{\bm{\varphi}}
\def \bpsi{\bm{\psi}}
\def \bnu{\bm{\nu}}
\def \bxi{\bm{\xi}}
\def \bomega{\bm{\omega}}

\def \bell{\bm{\ell}}

\def\eqref#1{(\ref{#1})}

\def\vec#1{\mathbf{#1}}

\newcommand{\abs}[1]{\left| #1 \right|}

\def\Sp{\mathbb{S}}

\def\Zq{\mathbb{Z}_q}
\def\Zqx{\mathbb{Z}_q^*}
\def\Zd{\mathbb{Z}_d}
\def\Zdx{\mathbb{Z}_d^*}
\def\Zf{\mathbb{Z}_f}
\def\Zfx{\mathbb{Z}_f^*}
\def\Zp{\mathbb{Z}_p}
\def\Zpx{\mathbb{Z}_p^*}
\def\cM{\mathcal M}
\def\cE{\mathcal E}
\def\cH{\mathcal H}

\def\le{\leqslant}

\def\ge{\geqslant}

\def\sfB{\mathsf {B}}
\def\sfC{\mathsf {C}}
\def\L{\mathsf {L}}
\def\FF{\mathsf {F}}

\def\sE {\mathscr{E}}
\def\sS {\mathscr{S}}

%%%%%%%%%%%%%%%%%%%%%%%%%
% Alphabet calligraphie %
%%%%%%%%%%%%%%%%%%%%%%%%%
\def\cA{{\mathcal A}}
\def\cB{{\mathcal B}}
\def\cC{{\mathcal C}}
\def\cD{{\mathcal D}}
\def\cE{{\mathcal E}}
\def\cF{{\mathcal F}}
\def\cG{{\mathcal G}}
\def\cH{{\mathcal H}}
\def\cI{{\mathcal I}}
\def\cJ{{\mathcal J}}
\def\cK{{\mathcal K}}
\def\cL{{\mathcal L}}
\def\cM{{\mathcal M}}
\def\cN{{\mathcal N}}
\def\cO{{\mathcal O}}
\def\cP{{\mathcal P}}
\def\cQ{{\mathcal Q}}
\def\cR{{\mathcal R}}
\def\cS{{\mathcal S}}
\def\cT{{\mathcal T}}
\def\cU{{\mathcal U}}
\def\cV{{\mathcal V}}
\def\cW{{\mathcal W}}
\def\cX{{\mathcal X}}
\def\cY{{\mathcal Y}}
\def\cZ{{\mathcal Z}}
\newcommand{\rmod}[1]{\: \mbox{mod} \: #1}

\def\cg{{\mathcal g}}

\def\vy{\mathbf y}
\def\vr{\mathbf r}
\def\vx{\mathbf x}
\def\va{\mathbf a}
\def\vb{\mathbf b}
\def\vc{\mathbf c}
\def\ve{\mathbf e}
\def\vh{\mathbf h}
\def\vk{\mathbf k}
\def\vm{\mathbf m}
\def\vz{\mathbf z}
\def\vu{\mathbf u}
\def\vv{\mathbf v}
\def\v0{\mathbf 0}

\def\ee{{\mathbf{\,e}}}
\def\eep{{\mathbf{\,e}}_p}
\def\eem{{\mathbf{\,e}}_m} %\em already exists and is used for \emph

\def\Tr{{\mathrm{Tr}}}
\def\Nm{{\mathrm{Nm}}}

 \def\SS{{\mathbf{S}}}

\def\lcm{{\mathrm{lcm}}}

 \def\0{{\mathbf{0}}}

\def\({\left(}
\def\){\right)}
\def\fl#1{\left\lfloor#1\right\rfloor}
\def\rf#1{\left\lceil#1\right\rceil}
\def\sumstar#1{\mathop{\sum\vphantom|^{\!\!*}\,}_{#1}}

\def\mand{\qquad \mbox{and} \qquad}

\def\tblue#1{\begin{color}{blue}{{#1}}\end{color}}

%%%%%%%%%%%%%%%%%%%%%%%%%%%%%%%%%%%%%%%%%%%%%%%%%%%%%%%%
%%%%%%%%%%%%%%%%%%%%%%%%%%%%%%%%%%%%%%%%%%%%%%%%%%%%%%%%
%%%%%%%%%%%%%%%%%%%%%%%%%%%%%%%%%%%%%%%%%%%%%%%%%%%%%%%%
%%%%%%%%%%%%%%%%%%%%%%%%%%%%%%%%%%%%%%%%%%%%%%%%%%%%%%%%

%%%%%%%  END OF STANDARD STUFF %%%%%%%%%

%%%%%%%%%%%%%%%%%%%%%%%%%%%%%%%%%%%%%%%%%%%%%%%%%%%%%%%%
%%%%%%%%%%%%%%%%%%%%%%%%%%%%%%%%%%%%%%%%%%%%%%%%%%%%%%%%
%%%%%%%%%%%%%%%%%%%%%%%%%%%%%%%%%%%%%%%%%%%%%%%%%%%%%%%%
%%%%%%%%%%%%%%%%%%%%%%%%%%%%%%%%%%%%%%%%%%%%%%%%%%%%%%%
%%%%%%%%%%%
%%% Spell

\hyphenation{re-pub-lished}

\mathsurround=1pt

\def\bfdefault{b}

\def \F{{\mathbb F}}
\def \K{{\mathbb K}}
\def \N{{\mathbb N}}
\def \Z{{\mathbb Z}}
\def \P{{\mathbb P}}
\def \Q{{\mathbb Q}}
\def \R{{\mathbb R}}
\def \C{{\mathbb C}}
\def\Fp{\F_p}
\def \fp{\Fp^*}

 \def \xbar{\overline x}

\title{Chowla and Sarnak Conjectures for Kloosterman Sums}

\author[E.~H.~El Abdalaoui]{El Houcein El Abdalaoui}
\address{
Laboratoire de Math{\'e}matiques Rapha{\"e}l Salem, 
Universit{\'e} de Rouen Normandie,  Saint-{\'E}tienne-du-Rouvray, 
F76801, France}
\email{elhoucein.elabdalaoui@univ-rouen.fr}

 \author[I. E. Shparlinski] {Igor E. Shparlinski}

\address{Department of Pure Mathematics, University of New South Wales,
Sydney, NSW 2052, Australia}
\email{igor.shparlinski@unsw.edu.au}

 \author[R. S. Steiner] {Raphael S. Steiner}
 \address{Department of Mathematics,    Eidgen{\"o}ssische Technische Hochschu\-le, Z{\"u}rich, 8092, Switzerland}
\email{raphael.steiner@math.ethz.ch}

\begin{abstract} 
We formulate several analogues of the Chowla and  Sarnak conjectures, 
which are widely known in the setting of the M{\"o}bius function, in the setting of Kloosterman sums. 
We then show that for Kloosterman sums, in some cases, these conjectures can be established 
unconditionally. 
\end{abstract} 

\keywords{Chowla  conjecture,  Sarnak conjecture, Kloosterman sum}
\subjclass[2020]{11L07,  11T23, 37B40}

\maketitle

\tableofcontents

\section{Introduction}

\subsection{Background and motivation}

Recently,  there has been a lot of activity related to the  Chowla and 
Sarnak conjectures for the 
{\it M{\"o}bius function\/} $\mu(n)$. 
We recall that these conjectures assert  non-correlation  between 
shifted values of $\mu(n)$ and  between  $\mu(n)$   and 
low complexity sequences, respectively,  %% see~\cite{, 
 see~\cite{elAb, elAbNe, elabAM, AS, BSZ,  FKL,HeRa,LiSa,Mul,MRT,Sarnak,SaUb,Tao, TaTe} and references therein. Both conjectures are special cases of the general {\it M{\"o}bius Randomness Law\/}, formulated, for example, 
in~\cite[Section~13.1]{IwKow}.

Here, we introduce and investigate similar conjectures for  {\it Kloosterman sums\/}, 
see~\cite[Equation~(1.56)]{IwKow}.  Generally, these conjectures appear to be much harder for
Kloosterman sums because of:
\begin{itemize}
\item[(i)] {\it lack of multiplicativity\/} between Kloosterman sums;
\item[(ii)]  a {\it dense set\/} of their values rather a discrete set \{-1,0,1\} as in the case of the M{\"o}bius function.
\end{itemize}
Nevertheless, we establish them in several special cases. 

In fact, some special instances of what one may call {\it Kloosterman Randomness Law\/}
is already known, see Section~\ref{sec:Rand-K} below. Moreover,  quite surprisingly, in some cases,
for Kloosterman sums we are able to establish results which are superior to those known for the M{\"o}bius 
function. 
 
\subsection{Set-up}
We denote the residue ring modulo $m$ by $\Z_m$ 
and denote the group of units of $\Z_m$ by $\Z_m^*$. 

It is more convenient to work with  {\it normalised  Kloosterman sums\/}, 
which, for integers $a$ and $m \ge 1$, we define as
$$
\Kma=\frac{1}{\sqrt{m}} \sum_{x \in \Z_m^*} \eem\(ax +\xbar \),
$$
where $\xbar$ is the multiplicative inverse of $x$ modulo $m$ and
$$
\eem(z) = \exp(2 \pi i z/m).
$$
We shall further simply write $\ee(z)$ for $\ee_1(z)$.
%and $\tau(m)$ is the number of positive divisors of $m$.
By Weil~\cite{Weilbound} and Estermann~\cite{EstBound}, we have the following bound on Kloosterman sums
\begin{equation}
\label{eq:Weil}
|\Kma|\le  2^{\omega(m)} \begin{cases}1, & 2^5 \! \nmid \ m,\\ 2^{1/2}, & 2^5 \| m, \\ 2, & 2^6 \| m, \\ 2^{3/2}, & 2^7 \mid m, \end{cases}
\end{equation}
where $\omega(n)$ is the number of distinct primes factors in $n$, cf.~\cite[Ch. 9]{KnightlyLi},
and for a prime $p$ and an integer $k\ge 1$, we use $p^k \| m$ to denote that $p^k\mid m$ but
$p^{k+1} \nmid m$. 
%%see~\cite[Corollary~11.12]{IwKow}. 
It is also easy to see that the values of $\Kma$ are real. Indeed, 
$$
\Kma=\frac{1}{\sqrt{m}} \sum_{x \in \Z_m^*} \eem\(ax +\xbar \) =
\frac{1}{\sqrt{m}} \sum_{x \in \Z_m^*} \eem\(-ax - \xbar \) = \overline{\Kma}
$$
and thus $\Kma \in \R$.

Since Kloosterman sums depend on two parameters $a$ and $m$, following standard 
terminology, one can study
\begin{itemize}
\item {\it  horizontal  randomness\/}, that is, one aims to show that the sums of the types
$$
\sum_{m=1}^M a_m \Kma  \mand \sum_{m=1}^M \cK_{m+h_1} (a) \ldots \cK_{m+h_s} (a)
$$
are small compared to their trivial bound $M^{1+o(1)}$, which follows from~\eqref{eq:Weil}, 
for some ``interesting'' (bounded)  sequence $a_m$, $m =1, \ldots, M$, and some integer shifts $0 \le h_1< \ldots < h_s$. 

\item {\it  vertical randomness\/}, that is, one aims to show that the sums of the types
\begin{equation}
\label{eq:H-PRN}
\sum_{n =1}^N b_n \cK_{m}(n)  \mand \sum_{n =1}^N \cK_{m} (n+h_1) \ldots  \cK_{m} (n+h_s)
\end{equation}
are small compared to their trivial bound $N M^{o(1)}$, which follows from~\eqref{eq:Weil}, 
for some ``interesting'' (bounded) arithmetic  sequence $b_n$, $n =1, \ldots, N$, and some integer shifts $0 \le h_1< \ldots < h_s< q$. 
\end{itemize}

It is especially interesting to estimate the above sums with a power saving and obtain estimates of the type
$O(M^{1-\eta})$ and $O(NM^{-\eta})$, respectively, for some positive constant $\eta>0$. 

 \subsection{Notation}
\label{sec:not}

Throughout the paper, the notation $U = O(V )$, 
$U \ll V$ and $ V\gg U$  are equivalent to $|U|\leqslant c|V| $ for some positive constant $c$, 
which throughout the paper may depend on the degree $d$ and occasionally on the small real positive 
parameters $\varepsilon$ and $\delta$.  

For any quantity $V> 1$, we write $V^{o(1)}$ (as $V \to \infty$) to indicate a function of $V$ which does not exceed $V^\eps$ for any $\eps> 0$, provided $V$ that is large enough. The conjunct notation $U \ll V^{o(1)}$ is thus subsequently to be interpret as for any $\eps$, there is a $V_0=V_0(\eps)>1$ and a $c=c(\eps, V_0)>0$, such that $U \le cV^{\eps}$ for all $V \ge V_0$.

%(which deviates slightly from the canonical  definition of $o(1)$).

More generally, when we say that a certain parameter is {\it fixed\/} this means that we allow all implied
constants to depend on this parameter.

%For a complex number $z$ we denote
%$$
%\e(z) = \exp(2\pi i z).
%$$

%R: this was moved to earlier.
 
As usual, we use $\mu(m)$ to denote the M{\"o}bius function, that is, $\mu(m) = (-1)^s$ when $m$ is a product of 
$s$ distinct primes and $\mu(m) =0$ otherwise. 
 
We use $\# \cA$ to denote the cardinality of a set $\cA$.  

Throughout the paper, $p$ always denotes a prime number. 
 
We always follow the following {\it conventions\/}: 
\begin{itemize}
\item  When we write $\Kma$, we always assume that $a$ is \emph{fixed} and thus the implied constant in `$\ll$' 
and similar expressions may depend on $a$ (and hence ~\eqref{eq:Weil} simply implies $|\Kma| \le m^{o(1)}$).
\item  When we write $ \cK_{m} (n)$, or  $\cK_{m} (n+h)$, we 
assume that $n$ is a parameter varying in some interval $[1,N]$, where $N$ may grow as fast as some power of $m$.
\end{itemize}
Given an number theoretic function $\xi: \N \to \C$, we define the {\it horizontal\/} averages
\begin{equation}
\sfH _{a,\xi}(M) =   \sum_{m =1}^M \xi(m) \cK_m(a)  
\label{eq:Haxi}
\end{equation}
and the {\it vertical} averages
$$
\sfV_\xi(m;N) =  \sum_{n =1}^N \xi(n) \cK_m(n).
$$
These are our main objects of study. 

It is also convenient to define the sums 
\begin{equation}
\overline\sfH_{a}(M) =  \sum_{m =1}^M    |\cK_m(a) | \mand 
\overline\sfV(m;N) =  \sum_{n =1}^N  |\cK_m(n) | , 
\label{eq:HVbar}
\end{equation}
which we are going to use a bench-marks for our estimates on $\sfH_{a,\xi}(M)$ and  $\sfV_\xi(m;N)$. 

It is certainly natural to expect that 
$$
 \overline\sfH_{a}(M) =  M^{1+ o(1)}   \mand  \overline\sfV(m;N) = N^{1+ o(1)}
$$
in wide range of parameters. 
For example,   by~\cite[Theorem~1.2]{FM}, for any fixed $r$  we have 
$$
M \frac{\(\log \log M\)^r}{ \log M} \ll  \overline\sfH_{1}(M)   \ll M \frac{\(\log \log M\)^{2 - 8/(3\pi)}}{\(  \log M \)^{1 - 8/(3\pi)}}, 
$$
and perhaps a similar result also holds for $\overline\sfH_{a}(M) $ for any fixed integer $a\ne 0$.

The bound~\eqref{eq:Weil} further suggests to define
$$\cK_m^{*}(a)=\frac{|\cK_m(a) |}{2^{\omega(m)}},$$
where $\omega(n)$ is the number of distinct primes factors in $n$.  
Thus, in particular, for an integer $a\ne 0$
\begin{equation}\label{eq:K* < 1}
\cK_m^{*}(a) \le 2^{3/2}, 
\end{equation}
Analogously to~\eqref{eq:Haxi} and~\eqref{eq:HVbar}, we define $\sfH_{a,\xi}^*(M)$ and $\overline\sfH_{a}^*(M)$. According to Fouvry and Michel~\cite[Theorem~1.1]{FM}, for any fixed $r \geq 1$ we have
\begin{equation}\label{FM-C2}
M \frac{\(\log \log M\)^r}{ \log M} \ll  \overline\sfH_{1}^*(M)   \ll M\(\frac{\log \log M}{\log M}\)^{1 - 4/(3\pi)}. 
\end{equation}
%where
%$$H_1^*(M)=  \sum_{m =1}^M  \cK_m^*(a).$$

%We define also
%$$
%\sfH _{a,\xi}^*(M) =   \sum_{m =1}^M \xi(m) \cK_m^*(a).$$

\section{Main results}

\subsection{Previous results}
\label{sec:Rand-K}

When it comes to {\it horizontal randomness\/}, only a few techniques have been successfully applied. The one that stands out is, of course, the use of Kuznetsov's trace formula. Kuznetsov~\cite{Kuz} developed said formula to prove a strong bound towards the Linnik--Selberg conjecture on sums of Kloosterman sums (see also~\cite[Section~16.1]{IwKow})
%Little is known about the {\it horizontal randomness\/}, with the notable exception of Kuznetsov's result~\cite{Kuz} towards the Linnik--Selberg conjecture on sums of Kloosterman sums (see also~\cite[Section~16.1]{IwKow})
\begin{equation}
	\sfH _{a,\mathbf 1}(M) \ll M^{2/3+o(1)}, \quad a\ge 1, \text{ fixed},
	\label{eq:Linnik-Selberg-Kuz}
\end{equation}
where $\mathbf 1$ indicates the constant weight $\xi(m)=1$. Similar results stemming from more general Kuznetsov formul{\ae} were subsequently derived, see for example~\cite{DesIw, Drap,BlMil,DrMa,FM2,GaSe, KiralYoung, Stein-Twist-Linnik}. These cover horizontal averages against sequences of the type  $\xi(m)=1$ if $q\mid m$ and  $\xi(m)=0$ otherwise
% $$
% \xi(m)= \begin{cases}1, & \text{if } q \mid m, \\ 
%0, & \text{otherwise}, \end{cases}
%$$
%%% $1_{q \mid m}$, 
 for a fixed integer $q$; $\xi=\chi$, a fixed Dirichlet character; or mixtures thereof.

%%\commR{I've added an equation to the horizontal averages to make them 
%%appear more important (there are many vertical results cited)}
%%\commI{Referenced it later to justify the number, also there are several more changes, e.g. 
%%$1 \to \mathbf 1$ + explanations.} 

%Unfortunately essentially nothing is known about the {\it horizontal randomness\/}, 
%with the only exception of the sequence $a_m=1/\sqrt{m}$ in the first sum in~\eqref{eq:V-PRN}, given by 
%the celebrated results of Kuznetsov~\cite{Kuz}  towards  the
%Linnik--Selberg conjecture, see also~\cite[Section~16.1]{IwKow}.

In the direction of {\it vertical randomness\/}, several results and techniques are known, often in the most interesting case of prime $m = p$.

%There are, however, several results in the direction of  {\it vertical randomness\/}, typically  in the most interesting case of prime $m = p$.

\begin{itemize}[leftmargin= 0.5cm]

\item  {\it Correlations between shifted values:} For a prime $m = p$, general results of  Fouvry,   Kowalski and  
Michel~\cite[Corollary~1.6]{FKM2} contain as a special case 
 a bound on the second sum   in~\eqref{eq:H-PRN}.  Furthermore, Fouvry, Michel,  Rivat and   S{\'a}rk{\"o}zy~\cite[Theorem~1.1]{FMRS} have estimated a variant of 
the second sum in  in~\eqref{eq:H-PRN} with the product of the 
sign-functions $\sign \cK_p(n+h_1)\ldots \sign \cK_p(n+h_s)$ instead of the sums themselves. 

\item {\it Correlations with some arithmetic functions:}
Fouvry,  Kowalski and Michel~\cite[Theorem~1.7]{FKM1}  and 
Kowalski, Michel and Sawin~\cite[Corollary~1.4]{KMS2} have given bounds
$$
\sfV_\mu(m;N)  \ll N^{1-\eta}   \mand \sfV_\tau(m;N)  \ll N^{1-\eta}
$$ 
with the M{\"o}bius $\mu(n)$ and divisor $\tau(n)$ functions provided that for some 
fixed $\varepsilon> 0$ we have
$$
N \ge p^{3/4  + \varepsilon} \mand  N \ge p^{2/3  + \varepsilon} ,
$$
respectively, where $\eta > 0$ depends only on $\varepsilon > 0$. Perhaps the argument 
of the proof of~\cite[Theorem~1.8]{BFKMM2},   can  be  used to improve the dependence $\eta$ on 
$\varepsilon$ for sums with $\mu(n) \cK_p(n) $ and thus improve the bound of~\cite[Theorem~1.7]{FKM1}, however it is not likely to help to extend any of  the above ranges.
Korolev and Shparlinski~\cite[Theorems~2.1 and~2.2]{KorShp}, using a different method, have obtained nontrivial bounds on both sums  already for 
$$
 N \ge p^{1/2  + \varepsilon} ,
$$
however the saving is only logarithmic. 
All these methods also apply in broader generality to sums with other arithmetic functions. 

It is also quite plausible that the results and methods  of~\cite{LSZ1, LSZ2} are able to produce 
estimates on similar sums modulo a prime power $m = p^k$ for a fixed $p$ and growing $k$, and 
in fact it is reasonable to expect that these results are  nontrivial starting with very small 
values of $N$, for example,  for 
$$
 N \ge p^{\varepsilon}. 
$$

\item {\it Correlations with some digital  functions:} For any integers $0\le s\le r$,  let $\chi_{r,s}(n)$ be the characteristic function 
of   the set $\cG_s(r) $ of $r$-bit integers with only   $s$ nonzero binary digits, 
thus 
$$
 \#  \cG_s(r) = \binom{r}{s}.
 $$
Korolev and Shparlinski~\cite[Section~9]{KorShp} have shown that if 
$$
2^r = p^{1+o(1)} \mand r/2 \ge s \ge (\rho_0 + \delta) r, 
$$
where $\rho_0 = 0.11002786\ldots$  is the root of the equation \vskip-9pt
$$
H(\vartheta) = 1/2, \qquad 0 \le \vartheta \le 1/2,
$$
with the {\it binary entropy function\/}  \vskip-9pt
$$
H(\gamma) = \frac{- \gamma \log \gamma  -  (1-\gamma)
\log (1-\gamma)}{\log 2}.
$$
Then, for any $\varepsilon > 0$, there exists $\eta > 0$ such that, 
$$
\sum_{n=1}^{2^r-1} \cK_p(n) \chi_{r,s}(n) \ll \binom{r}{s}^{1-\eta}.
$$

\item {\it Bilinear correlations:} The results of~\cite{BFKMM1,KSWX,KMS1,KMS2,Shp,ShpZha} 
give various bounds on bilinear sums with Kloosterman sums 
$$
 \sum_{k=1}^K  \sum_{n =1}^N \alpha_k \cK_m(kn),  
\mand \sum_{k=1}^K  \sum_{n =1}^N \alpha_k\beta_n \cK_m(kn),
$$
which are known as {\it Type~I\/} and {\it Type~II\/} sums, with some complex weights. 
Bounds of such sums, besides being other instances of Kloosterman randomness, are also important for many applications, see~\cite{BFKMM1,BFKMM2,KeSh,KMS1}.

\end{itemize}

\subsection{New results in the horizontal aspect}
It is natural to assume that both $ \cK_m(a)$ and   $ \cK_m^*(a)$ and the M{\"o}bius function function $\mu(m)$ 
are not correlated.   

\begin{conj}
\label{conj:K-mu}  For any fixed integer $a \ne 0$, we have 
$$
\sfH_{a,\mu}(M) =  o \(\overline\sfH_{a}(M)\) \mand \sfH_{a,\mu}^*(M) =  o \(\overline\sfH_{a}^*(M)\)
$$
as $M\to \infty$.
\end{conj}  

It is quite possible that the convergence to zero in Conjecture~\ref{conj:K-mu} is quite fast 
and in fact 
$$
\sfH_{a,\mu}(M)  \ll M^{-\eta} \overline\sfH_{a}(M)
\mand \sfH_{a,\mu}^*(M)  \ll M^{-\eta} \overline\sfH_{a}^*(M), 
$$
for some constant $\eta>0$.

This leads us to formulating the following versions of the {\it Chowla conjecture\/}, see~\cite{Sarnak,Tao}, for 
the  Kloosterman sums. We note that since we still do not know the exact order of 
magnitude of $\overline\sfH_{1}^*(M)$, see~\eqref{FM-C2}, we need more explicit 
bounds on the assumed rate of decay than in the classical Chowla conjecture. 

\begin{conj}
	\label{conj:K-Chowla-II}  For any fixed integer $a \ne 0$,  and  any fixed positive integers $\nu_1, \ldots, \nu_s$
	and  integers  $ h_s > \ldots > h_1 \ge 0$, as  $M\to \infty$, we have:
\begin{itemize}
\item if  $\nu_1, \ldots, \nu_s$ are  not  all  even	, then 
$$
	\left| \frac{1}{ M  } \sum_{m=1}^M \cK_{m+h_1}^*(a)^{\nu_1} \ldots  \cK_{m+h_s}^*(a)^{\nu_s} \right |  
	= o\(\frac{\overline\sfH_{a}^*(M)  } {M}\)^{\nu_1 + \ldots + \nu_s}	
$$
\item if  $\nu_1, \ldots, \nu_s$ are    all  even, then 
$$
	\left| \frac{1}{ M  } \sum_{m=1}^M \cK_{m+h_1}^*(a)^{\nu_1} \ldots  \cK_{m+h_s}^*(a)^{\nu_s} \right |  
	 \le \(A \frac{\overline\sfH_{a}^*(M)  } {M}\)^{\nu_1 + \ldots + \nu_s}	. 
$$
with some constant $A \ge 1$ which may  depend only upon $a$ and is uniform with 
respect to all other parameters. 
\end{itemize}
\end{conj}

%%%\begin{conj}
%%%	\label{conj:K-Chowla-II}  For any fixed integer $a \ne 0$,  and  any fixed positive integers $\nu_1, \ldots, \nu_s$,  
%%%\begin{itemize}
%%%\item if  $\nu_1, \ldots, \nu_s$ are  not  all  even	, then 
%%%$$
%%%	\left| \frac{1}{ M  } \sum_{m=1}^M \cK_{m+h_1}^*(a)^{\nu_1} \ldots  \cK_{m+h_s}^*(a)^{\nu_s} \right |  
%%%	= o\(\frac{\overline\sfH_{a}^*(M)  } {M}\)^{\nu_1 + \ldots + \nu_s}	
%%%$$
%%%\item if  $\nu_1, \ldots, \nu_s$ are    all  even, then 
%%%$$
%%%	\left| \frac{1}{ M  } \sum_{m=1}^M \cK_{m+h_1}^*(a)^{\nu_1} \ldots  \cK_{m+h_s}^*(a)^{\nu_s} \right |  
%%%	 \le \(A \frac{\overline\sfH_{a}^*(M)  } {M}\)^{\nu_1 + \ldots + \nu_s}	. 
%%%$$
%%%with some constant $A \ge 1$ which may depend only upon $a$, 
%%%\end{itemize}
%%%for any fixed integers  $ h_s > \ldots > h_1 \ge 0$ and $M\to \infty$. 
%%%\end{conj}  

We note that the second part of Conjecture~\ref{conj:K-Chowla-II} is still nontrivial (which makes
a remarkable difference between Conjecture~\ref{conj:K-Chowla-II} and the  Chowla
conjecture for the  M{\"o}bius function). 

Next, under Conjecture~\ref{conj:K-Chowla-II}, we show that $\sfH_{a,\xi}^*(M) =  o\(\overline\sfH_{a}^*(M)  \)$ for a natural class of 
``low complexity'' sequences.  

First, we  introduce the notion of {\it topological entropy\/} following works of Bowen~\cite{Bow1,Bow2} 
and Dinaburg~\cite{Din}.  

For a compact metric set  $\cX$ and a  homeomorphism  $T$ on it, we consider 
 a topological metric space  $(\cX,T)$. 

We define the  distance on $\cX$ at step $n$ by
\begin{equation}\label{eq:dn-Def}
d_n(x,y) = \max_{0 \le k \le n-1} d(T^{k} x ,T^{k}y).
\end{equation}
We say that a set $\cS \subseteq \cX$ is $(n, \varepsilon)$-separated if for all $x, y \in \cS$ with $x\ne y$, we have $d_n(x,y)\geq \varepsilon$.
Since $\cX$ is compact, a separated set cannot be infinite and we can define $s(n,\varepsilon)$ to be the largest cardinalty of an  $(n, \varepsilon)$-separated set. Similarly, we say that a set $\cR \subseteq \cX$ is an $(n,\varepsilon)$-span set if 
$$
\cX \subseteq \bigcup_{x \in \cR} \sfB_{d_n}(x,\varepsilon), 
$$ 
where $\sfB_{d_n}(x,\varepsilon)$ is the ball  of radius   $ \varepsilon$ with respect to $d_n$: 
$$
 \sfB_{d_n}(x,\varepsilon) = \{y \in \cX:~ d_n(x,y) \le \varepsilon\}
$$ 
By compactness,
there are finite $(n,\varepsilon)$-spanning sets. Let $r(n,\varepsilon)$  be the minimum cardinality of the $(n, \varepsilon)$-spanning sets. This number corresponds to the minimal number of points in $\cX$ such that each orbit of length $n$ can be $\varepsilon$-approximated by the orbit of one of these points. 

It is easy to see that these two numbers are linked by the following inequality for $ n \in \mathbb{N}$ and $\varepsilon>0$:
$$ r(n, \varepsilon/2) \le s(n, \varepsilon) \le r(n,\varepsilon).
$$
Following~\cite{Bow1}, the topological entropy of $T$ is defined by  
$$
h_{\mathrm{top}}(T)= \lim_{\varepsilon \to 0} \limsup_{n \to \infty} \frac{1}{n}\log s(n, \varepsilon)
=\lim_{\varepsilon \to 0} \limsup_{n \to \infty}\frac{1}{n}\log r(n, \varepsilon).
$$ 
We introduce also the notion of dynamical sequence. 
\begin{definition}
\label{def:det func}
	The  function $\xi: \N \to \C$ is said to be deterministic if there exists a dynamical system $(\cX,T)$  of 
	topological entropy  zero, a continuous function $f:\cX \to \C$ and a point $x\in \cX$ such that for all $n \in \mathbb{N}$  we have $\xi(n) = f\(T^n x\)$.
\end{definition}

\begin{theorem}
\label{conj:K-S}   Under Conjecture~\ref{conj:K-Chowla-II},  for any deterministic function   $\xi: \N \to \C$  and a fixed integer $a \ne 0$ we have
$$
\sfH_{a,\xi}^*(M) =  o\(\overline\sfH_{a}^*(M)  \). 
$$
\end{theorem}  

For the proof of Theorem~\ref{conj:K-S}, we follow combinatorial ideas of Sarnak~\cite{Sarnak} and Tao~\cite{Tao}.

%We now have the following conditional result towards a version of the Sarnak 
%conjecture for Kloosterman sums.   

%%\commI{Referenced \eqref{eq:Linnik-Selberg-Kuz} here, some other small changes} 
As a further evidence of validity of Conjecture~\ref{conj:K-mu}, we now obtain variants of~\eqref{eq:Linnik-Selberg-Kuz} for other weights of arithmetic nature. 
For example, we give a result about 
non-correlation of Kloosterman sums and the characteristic function $\kappa_k(m)$ of $k$th-power 
free numbers (that is, numbers indivisible by $k$th power of a prime), in particular, 
$$
\kappa_2(m) = |\mu(m)|
$$
and the Euler function $\varphi(m)$. In order to state the theorem, we first need to introduce $\vartheta$ as the currently best bound towards the {\it Selberg eigenvalue conjecture\/}. That is the cuspidal spectrum of the (negated) hyperbolic Laplacian on congruence quotients $\Gamma(N) \backslash \mathbb{H}$ is lower bounded by $\frac{1}{4}-\vartheta^2$. Note that we have
$$
\vartheta \le \frac{7}{64}
$$ 
by a result of Kim and Sarnak~\cite{KiSa}, whilst the {\it Selberg eigenvalue conjecture\/} states that one should have $\vartheta = 0$. We now define
\begin{equation}
	\label{eq:gamma}
	\gamma=\max\{1/6,2 \vartheta\}. 
\end{equation}
In particular, we see that $\gamma \le 7/32$.

\begin{theorem}
\label{thm:K k-free}  For any fixed  integer $a$, and $k\ge 2$, 
we have 
$$
\sfH_{a,\kappa_k}(M)  \le M^{1/2+\gamma+o(1)} +  M^{1/2+1/(2k)+o(1)} .
$$
\end{theorem}

Next we show that Kloosterman sums do not correlate with  the Euler function $\varphi(m)$.

\begin{theorem}
\label{thm:K phi}  For any fixed integer $a$,  with $\gamma$ given by~\eqref{eq:gamma},  we have 
$$
\sfH_{a,\varphi}(M)  \le  M^{3/2+\gamma+o(1)}. 
$$
\end{theorem}

\subsection{New results in the vertical  aspect}
\label{sec:vert} 

First, similarly to~\cite[Definition~1.3]{FKM2}, we say that a vector $(h_1, \ldots, h_s) \in \Z^s$ is {\it normal modulo $p$}
if there is some $h$ such that 
$$
\#\{j:~1 \le j \le s, \ h_j \equiv h \pmod p\} \equiv 1 \pmod 2.
$$

\begin{theorem}
\label{thm:K-LinPoly} 
For a prime $p$ and integer $N < p$,   uniformly over   normal modulo $p$ vectors $(h_1, \ldots, h_s) \in \Z^s$ and polynomials $g \in \R[X]$ of degree $d\ge 0$, we have 
$$
\sum_{n=1}^N  \cK_p(n+h_1)   \ldots  \cK_p(n+h_s) \ee\(g(n)\) \ll N^{1- 2^{-d}}p^{2^{-d-1}}  ( \log p)^{2^{-d}} . 
$$
\end{theorem}

We also have the following unconditional analogue of Theorem~\ref{conj:K-S}

\begin{theorem}\label{thm:VertCorr-Incomp-II} Let $\psi(z)$ be a fixed arbitrary real function with $\psi(z) \to \infty$ 
as $z\to \infty$. For any dynamical system $(\cX,T)$ with zero topological entropy, for any continuous function $f:\cX \to \C$  and any $x \in \cX$, for a prime $p$, for 
an integer $N$ with $p > N \ge \sqrt{p} \psi(p) $  we have
	$$\sum_{n=1}^{N} \cK_p(n) f\(T^nx\)=o(N) $$
as $p\to \infty$.  
\end{theorem}

\section{Preliminaries} 

\subsection{Sums of Kloosterman sums in the horizontal  aspect}

We start by considering the special case of sums of Kloosterman sums along moduli which are multiples of a given integer $q \ge 1$. For $q=1$, cancellation along such sums were first conjectured by 
Selberg~\cite{Kloos-Selberg} and Linnik~\cite{Kloos-Linnik} (independently) and later proven by Kuznetsov~\cite{Kuz} with the eponymous trace formula. For general $q$, adaptations of the Kuznetsov trace formula and as well as cancellations along these sums have been demonstrated by, for example, Deshouillers and Iwaniec~\cite{DesIw}. For more uniform versions, the reader may wish to consult~\cite{GaSe, GaSe-Corr, SarTsim, Stein-Twist-Linnik} and for more general arithmetic progressions~\cite{BlMil, Drap, KiralYoung}.

Here, we slightly improve the $q$ dependence in the bound for the sums of Kloosterman sums along moduli divisible by $q$. We achieve this by applying a more optimised treatment of the exceptional spectrum. Otherwise, we follow the arguments of prior works \cite{DesIw,GaSe,GaSe-Corr,Stein-Twist-Linnik} almost verbatim.

%Here we a establish a result which complements that of Blomer and Mili{\'c}evi{\'c}~\cite{BlMil}, who have studied sums of Kloosterman sums  over moduli in a progression $m\equiv r \pmod q$ with relatively prime integers $q$ and $r$. However, we are interested in sums over moduli  $m\equiv 0 \pmod q$, to which neither their result nor their technique (based of detecting the condition $m\equiv r \pmod q$ via multiplicative characters)  apply. 

It is useful to record the following trivial bound (that is, when one forgoes any cancellation in the sum over $m$)
%%\footnote{Trivial, only in the sense that one forgoes any cancellation in the sum over $m$.}
\begin{equation}
\label{eq:Kloos-sums-triv}
\sum_{\substack{m =1\\ m\equiv 0 \pmod q}}^M \frac{1}{m^{1/2}}  \Kma \ll 
M^{1/2+o(1)} q^{-1}, 
\end{equation}
implied by a direct application of the Weil bound~\eqref{eq:Weil}.

\begin{lemma}\label{lem:K-AP} 
For a fixed non-zero integer $a$ and a positive integer $q$, with $\gamma$ given by~\eqref{eq:gamma},  
we have 
$$\sum_{\substack{m =1\\ m\equiv 0 \pmod q}}^M \frac{1}{m^{1/2}}  \Kma \ll 
\(Mq^{-2}\)^\gamma (qM)^ {o(1)}, 
$$
\end{lemma}

\begin{proof}  We observe that the trivial bound~\eqref{eq:Kloos-sums-triv} is stronger if $M \le q^2$. 
Hence, for the remainder of the proof, we always assume 
\begin{equation}
\label{eq:Large M}
M > q^2. 
\end{equation} 
It is sufficient to treat dyadic sums 
$$
S_\text{dyad}(M) = \sum_{\substack {M \le m < 2M\\q \mid m}} \frac{1}{m^{1/2}}\Kma, 
$$
which we compare  to smooth sums 
 \begin{equation}
\label{eq:Kloos-sums-smooth}
S_\text{smooth}(M) =
\sum_{m \equiv 0 \pmod q}\frac{1}{m^{1/2}} \Kma \, g\!\left(\frac{4\pi \sqrt{|a|}}{m}\right),
\end{equation}
where $g$ is a smooth bump function with
\begin{itemize}
	\item $g(x) =1$ for 
	$$\frac{2 \pi \sqrt{|a|}}{M} \le x \le \frac{4 \pi \sqrt{|a|}}{M},
	$$
	\item $g(x) = 0$ for 
	$$x \le \frac{2 \pi \sqrt{|a|}}{M+T}\qquad \text{or} \qquad 
	\frac{4 \pi \sqrt{|a|}}{M-T} \ge x,
	$$
	\item with $L_1$-norms satisfying 
	$$\|g'\|_1 \ll 1 \mand \|g''\|_1 \ll \frac{C^2}{\sqrt{|a|} T},
	$$
\end{itemize}
for some parameter $1 \le T \le M/2$. The Weil bound~\eqref{eq:Weil} shows that 
$$
\left|S_\text{dyad}(M)  - S_\text{smooth}(M)\right |  \ll M^{-1/2}(1+T/q)(Mq)^{o(1)}, 
$$
 see~\cite[Equation~(3.2)]{Stein-Twist-Linnik}.  
 
The smooth sum~\eqref{eq:Kloos-sums-smooth}, 
following the analysis~\cite[pp.~264--268]{DesIw},  may be bounded as 
$$
S_\text{smooth}(M) \ll \frac{M^{1/2}}{T^{1/2}}+\sum_{j, \text{ exc.}} |\rho_j(a)\rho_j(1)| M^{2|t_j|},
$$
where the sum is over an orthonormal basis of exceptional Maass forms of $\Gamma_0(q) \backslash \mathbb{H}$ with Fourier expansion
$$
y^{1/2} \sum_{n \neq 0} \rho_j(n) K_{i t_j}(2 \pi |n|y) e(nx) ,\qquad x+iy \in \mathbb{H}.
$$
We refer to~\cite{DesIw, Iwspec} for a background and standard notations. 
This exceptional term may be further bounded using
$$
|\rho_j(a)| \ll |a|^{1/2} |\rho_j(1)| \ll |\rho_j(1)|
$$
and also a density estimate for exceptional eigenvalues, see~\cite[Equation~(11.25)]{Iwspec}, along the lines of~\cite[Sections~2.1]{FM2} and~\cite[Lemma~2.10]{Topa}:
$$\begin{aligned}
\sum_{j, \text{ exc.}} |\rho_j(a)\rho_j(1)| M^{2|t_j|} &\ll (Mq^{-2})^{2\vartheta} \sum_{j, \text{ exc.}} |\rho_j(1)|^2 q^{4|t_j|} \\
&\ll (Mq^{-2})^{2\vartheta}q^{o(1)}.
\end{aligned}$$
In conclusion,
$$
S_\text{dyad}(M)  \ll \left(\frac{T}{qM^{1/2}} + \frac{M^{1/2}}{T^{1/2}} + (Mq^{-2})^{2\vartheta} \right) (qM)^{o(1)}.
$$
The choice $T  = 0.5 (qM)^{2/3}$  (which due to~\eqref{eq:Large M} ensures 
that $1 \le T \le M/2$, as required) gives the bound 
$$
S_\text{dyad}(M)  \ll \left(M^{1/6}q^{-1/3} + (Mq^{-2})^{2\vartheta} \right) (qM)^{o(1)}.
$$
The desired conclusion follows upon recalling the restriction~\eqref{eq:Large M}.
\end{proof}

Via partial summation, we further derive from Lemma~\ref{lem:K-AP} 

\begin{cor}\label{cor:K-AP} 
For a fixed non-zero integer $a$ and a positive integer $q$, with $\gamma$ given by~\eqref{eq:gamma},  we have for any fixed $\alpha \ge -1/2$ 
$$\sum_{\substack{m =1\\ m\equiv 0 \pmod q}}^M m^{\alpha} \Kma \ll 
M^{\alpha+1/2} \(Mq^{-2}\)^{\gamma}  (Mq)^ {o(1)}. 
$$
%and 
%$$\sum_{\substack{m =1\\ m\equiv 0 \pmod q}}^M  m\Kma \ll 
%M^{3/2} \(Mq^{-2}\)^{\gamma}  (Mq)^ {o(1)}.
%$$
\end{cor}

\subsection{Sums of Kloosterman sums in the vertical aspect}

We start with a  result which a special case of  Fouvry,  Kowalski and  Michel~\cite[Corollary~1.6]{FKM2}.
We also recall the definition of normal vectors from Section~\ref{sec:vert}.

\begin{lemma}\label{lem:H-Corr} 
For a prime $p$, for any integer $b$,  uniformly over   normal modulo $p$ vectors $(h_1, \ldots, h_s) \in \Z^s$,  we have
$$
\left|  \sum_{n=1}^p    \cK_p(n+h_1)   \ldots  \cK_p(n+h_s) \eep(b n) \right|
\ll p^{1/2} .
$$
\end{lemma}

Using the standard completion technique, see~\cite[Section~12.2]{IwKow}, 
we immediately derive from Lemma~\ref{lem:H-Corr} that for a prime $p$ and integer $N \ge 1$,  uniformly over   normal modulo $p$ vectors $(h_1, \ldots, h_s) \in \Z^s$, where $s$ is a fixed integer, we have
\begin{equation}\label{eq:Compl}
\left|  \sum_{n=1}^N    \cK_p(n+h_1)   \ldots  \cK_p(n+h_s)  \right|
\ll p^{1/2}  \log p. 
\end{equation}
However, we  need a different bound which is usually weaker than~\eqref{eq:Compl}
but instead is nontrivial for smaller values of $N$.

\begin{cor}\label{cor:H-Corr-Incomp} 
For a prime $p$ and integer $N <p$,  uniformly over   normal modulo $p$ vectors $(h_1, \ldots, h_s) \in \Z^s$, where $s$ is a fixed integer, 
 we have
$$
\left|  \sum_{n=1}^N    \cK_p(n+h_1)   \ldots  \cK_p(n+h_s)  \right|
\ll N^{1/2} p^{1/4}. 
$$
\end{cor} 

\begin{proof} Clearly for any integer $K\ge 1$ we have 
\begin{equation}\label{eq:Shift}
  \sum_{n=1}^N    \cK_p(n+h_1)   \ldots  \cK_p(n+h_s) =   \frac{1}{K} W + O(K) , 
 \end{equation}
 where
\begin{align*}
W & =  \sum_{k=1}^K \sum_{n=1}^N    \cK_p(n+k+h_1)   \ldots  \cK_p(n+k+h_s) \\
& = \sum_{n=1}^N    \sum_{k=1}^K   \cK_p(n+k+h_1)   \ldots  \cK_p(n+k+h_s) .  
\end{align*}  
By the Cauchy inequality, 
\begin{align*}
|W|^2 & \le N \sum_{n=1}^N \left |   \sum_{k=1}^K   \cK_p(n+k+h_1)   \ldots  \cK_p(n+k+h_s) \right|^2\\
 & \le N \sum_{n=1}^p \left |  \sum_{k=1}^K   \cK_p(n+k+h_1)   \ldots  \cK_p(n+k+h_s) \right|^2 \\
 & =  N \sum_{n=1}^p  \sum_{k, \ell =1}^K   \cK_p(n+k+h_1)   \ldots  \cK_p(n+k+h_s)\\
 & \qquad \qquad \qquad  \qquad  \qquad   \qquad  \cK_p(n+\ell+h_1)   \ldots  \cK_p(n+\ell+h_s)\\
  & =  N  \sum_{k, \ell =1}^K    \sum_{n=1}^p \cK_p(n+k+h_1)   \ldots  \cK_p(n+k+h_s)\\
 & \qquad \qquad \qquad  \qquad  \qquad   \qquad  \cK_p(n+\ell+h_1)   \ldots  \cK_p(n+\ell+h_s).
 \end{align*}
For $K<p$, there are at most $s^2 K$ pairs $(k,\ell)$ with $k-\ell \equiv h_i - h_j \mod(p)$ for some $1\le i, j \le s$ 
 (including $i=j$). In these occurrences, we estimate the inner sum trivially as $p$. For the remaining pairs  $(k,\ell)$, 
 the $2s$ integers $k+h_1,   \ldots , k+h_s, \ell+h_1   \ldots \ell+h_s$ are pairwise distinct modulo $p$ and so Lemma~\ref{lem:H-Corr} 
 applies. Hence, we derive
 $$
 W^2 \ll N(Kp + K^2 p^{1/2}). 
 $$
 We now choose $K = \rf{p^{1/2}}$ for which  $W \ll N^{1/2} p^{3/4}$, which after substitution in~\eqref{eq:Shift}
 implies 
\begin{equation}\label{eq:penult}
   \sum_{n=1}^N    \cK_p(n+h_1)   \ldots  \cK_p(n+h_s) \ll  N^{1/2} p^{1/4} + p^{1/2}.
 \end{equation}
  Clearly this bound is trivial for $N \le p^{1/2}$, while for   $N > p^{1/2}$ we see that  $ N^{1/2} p^{1/4}> p^{1/2}$ 
  and hence that term $p^{1/2}$ in~\eqref{eq:penult} can be discarded. 
\end{proof}

\section{Proofs of main results}

\subsection{Proof of Theorem~\ref{conj:K-S}} 
Let  $\xi(n) = f\(T^n x\)$ as in Definition~\ref{def:det func}. 
Since $\cX$ is compact, it  follows that $f$ is uniformly continuous by~\cite[Proposition~23]{Roy}.
This means for any $\varepsilon>0$ there is $\delta>0$ such that 
\begin{equation}\label{eq:UC}
\forall x, y \in \cX, \ d(x,y)  < \delta \Longrightarrow \left|f(x)-f(y)\right| <\varepsilon.
\end{equation}

Let $\varepsilon >0$ and set $\delta$ be as in~\eqref{eq:UC}. 

Next, we observe that for any fixed $h$, we have
$$
\sfH_{a,\xi}^*(M) =  \sum_{m=1}^{M} \cK_{m}^*(a) \xi_{m}=
 \sum_{m=1}^{M} \cK_{m+h}^*(a) \xi_{m+h}+O(h), 
$$
where throughout the proof all implied constants may depend only on $f$ and do not depend
on the parameters $\varepsilon$ and $H$ we introduce below.

We now fix some sufficiently large (in terms of $\varepsilon$) integer $H$ and note that by~\eqref{eq:K* < 1} we have 
\begin{equation}\label{eq:shifts}
\sfH_{a,\xi}^*(M)=
\frac{1}{H} \sum_{m=1}^{M} \sum_{h=1}^H \cK_{m+h}^*(a) \xi_{m+h}
+O(H),
\end{equation}
since $\xi_n=f\(T^nx\)$ and $f$ is a continuous function on the compact set $\cX$. 

We now argue as in~\cite[Section~2]{Tao} and note that  since  the topological entropy of $T$ is zero. 
We also recall the definition~\eqref{eq:dn-Def}. 

Then, for a sufficiently large $H$ (depending on $\varepsilon$), there is a set 
 $\{x_1, \ldots, x_t\} \subseteq \cX$  of   cardinality  
 \begin{equation}\label{eq:Bound t}
t \ll \exp\(\varepsilon^3 H\), 
\end{equation}
which spans $(\cX,T)$ by balls $B_{d_H}(x_i,\delta)$, $i =1, \ldots, t$,  of radius   $\delta$ with respect the metric induced by $d_H$. 
 That is, 
$$
\cX  \subseteq \bigcup_{i=1}^{t}B_{d_H}(x_i,\delta), 
$$ 
for some $x_1, \ldots, x_t \in \cX$ with  $t \le \exp\(\varepsilon^3 H\)$. 

In the other words, for any $x \in \cX$, and $m \ge 1$, there is 
$i_m \in \{1,\ldots,t\}$ for which we have
$$d\(T^h x_{i_m},T^{m+h}x\) \le \delta,$$
for all $1 \le h \le H$.  

 The rest of the argument, while is inspired by the exposition of Tao~\cite{Tao},
 deviates from that in~\cite[Section~2]{Tao}.  
 
Therefore, recalling~\eqref{eq:UC} we see that for all $1 \le h \le H$, we also have 
\begin{align*}
	\frac{1}{H}\sum_{h=1}^{H} \cK_{m+h}^*(a)f\(T^{m+h}x\)& =
	\frac{1}{H}\sum_{h=1}^{H} \cK_{m+h}^*(a)f\(T^hx_{i_m}\)\\
	& \qquad \qquad +O\(\frac{\varepsilon}{H}\sum_{h=1}^{H} |\cK_{m+h}^*(a)|\).
\end{align*}  
Hence, together with~\eqref{eq:shifts} we obtain
\begin{equation}\label{eq:shifts-2}
\sfH_{a,\xi}^*(M)=
W + O\(\Delta\), 
\end{equation}
where 
$$
W = \sum_{m=1}^{M} \frac{1}{H}  \sum_{h=1}^{H} \cK_{m+h}^*(a)f\(T^hx_{i_m}\), 
$$
and 
\begin{equation}\label{eq:Bound Delta}
\begin{split} 
\Delta 	& =  \sum_{m=1}^{M}  \frac{\varepsilon}{H}\sum_{h=1}^{H} |\cK_{m+h}^*(a)|= 
 \frac{\varepsilon}{H}\sum_{h=1}^{H}  \sum_{m=1}^{M}  |\cK_{m+h}^*(a)|\\
 & \ll   \frac{\varepsilon}{H}\sum_{h=1}^{H} \( \sum_{m=1}^{M}  |\cK_{m}^*(a)| + O(H)\)  \ll  \varepsilon \overline\sfH_{a}^*(M)  +  \varepsilon H \ll   \varepsilon \overline\sfH_{a}^*(M) , 
\end{split} 
\end{equation}
as $M\to \infty$.

We fix some integer $s\ge 1$ and, applying the H{\"o}lder inequality, derive
$$
|W |^{2s}  \le  H^{-2s} M^{2s-1} \sum_{m=1}^{M}\left | \sum_{h=1}^{H} \cK_{m+h}^*(a)f\(T^hx_{i_m}\)\right |^{2s} .
$$

It is now convenient to denote 
$$
\rho(M) = \frac{\overline\sfH_{a}^*(M)  } {M}.
$$
To estimate $W$, we eliminate the dependence of $x_{i_m}$ on $m$ with the trivial 
inequality 
$$
|W |^{2s}  \le   H^{-2s} M^{2s-1} \sum_{i=1}^t  \sum_{m=1}^{M}\left | \sum_{h=1}^{H} \cK_{m+h}^*(a)f\(T^hx_{i}\)\right |^{2s} .
$$
Since $\cX$ is compact, we can define 
$$
F = \sup_{x \in \cX} |f(x)| < \infty.
$$
We fix some integer $s\ge 1$ and, applying the H{\"o}lder inequality, derive
\begin{align*}
W^{2s} & \le  H^{-2s} M^{2s-1} \sum_{i=1}^t   \sum_{m=1}^{M}\left | \sum_{h=1}^{H} \cK_{m+h}^*(a)f\(T^hx_{i}\)\right |^{2s}\\
 & =  H^{-2s} M^{2s-1}  \sum_{i=1}^t  \sum_{h_1, \ldots h_{2s} =1}^H
\prod_{j=1}^s f\(T^{h_j} x_{i}\)\overline{ f\(T^{h_{s+j}} x_{i}\)}\\
&\quad \quad \quad \quad \sum_{m=1}^{M} \prod_{j=1}^s \cK_{m+h_j}^*(a) \cK_{m+h_{s+j}}^*(a)\\
 & \le   F^{2s}  H^{-2s} M^{2s-1} \sum_{i=1}^t   \sum_{h_1, \ldots h_{2s} =1}^H \left|
 \sum_{m=1}^{M} \prod_{j=1}^s \cK_{m+h_j}^*(a) \cK_{m+h_{s+j}}^*(a) \right| \\
 & = t  F^{2s}  H^{-2s} M^{2s-1}   \sum_{h_1, \ldots h_{2s} =1}^H \left|
 \sum_{m=1}^{M} \prod_{j=1}^s \cK_{m+h_j}^*(a) \cK_{m+h_{s+j}}^*(a) \right| .
\end{align*} 

Each tuple $\(h_1, \ldots h_{2s}\) \in [1, H]^{2s}$ that does not satisfy the conditions of Conjecture~\ref{conj:K-Chowla-II} can be split into $s$ pairs $h_j=h_k$, $j \ne k$. We thus see that their cardinality may be bounded by
$$
\binom{2s}{s}  s! H^s \le (sH)^s.
$$
Therefore, under Conjecture~\ref{conj:K-Chowla-II}, recalling~\eqref{eq:K* < 1}, we have
\begin{align*}
W^{2s}&  \le  t  F^{2s}  H^{-2s} M^{2s-1}   \((sH)^s  M  \(A\rho(M)\)^{2s}  +  o\(M\rho(M)^{2s} \)\)\\
& = t (AF)^{2s} M^{2s}  \rho(M)^{2s}  \( (s/H)^s    +  o(1 )\)\\
& =   (AF)^{2s}  M^{2s}  \rho(M)^{2s}  \( t(s/H)^s    +  o(1 )\)\\
& =   (AF)^{2s}    t(s/H)^s M^{2s}  \rho(M)^{2s} , 
\end{align*} 
as $M \to \infty$ (while $s$ and $H$ are fixed). 
Thus,
\begin{equation}\label{eq:Bound W prelim}
W    \ll   t^{1/2s} \(s/H\)^{1/2}   M  \rho(M)  =  t^{1/s} \(s/H\)^{1/2}  \overline\sfH_{a}^*(M),
\end{equation}
where the implies constant only depends on $a$ and $f$.

Choosing 
$$
s = \rf{\varepsilon^2 H}
$$
provided that $\varepsilon <1/2$. 
 assuming that $H$ is large enough so that $s \ge 2$, 
and recalling~\eqref{eq:Bound t} 
we see that 
$$
t^{1/s} \le  \exp\(\varepsilon\) \le 2 \mand   \(s/H\)^{1/2} \le 2 \varepsilon, 
$$
provided that $\varepsilon <1/2$. 
Therefore, with the above choice of $s$, the bound~\eqref{eq:Bound W prelim} becomes
\begin{equation}\label{eq:Bound W}
W    \ll    \overline\sfH_{a}^*(M)  \( \varepsilon    +  o(1)\)   \ll   \varepsilon  \overline\sfH_{a}^*(M) ,
\end{equation}
when $M \to \infty$. 

Substituting the bounds~\eqref{eq:Bound Delta} and~\eqref{eq:Bound W}
in~\eqref{eq:shifts-2} we obtain
$$
\sfH_{a,\xi}^*(M) \ll     \varepsilon  \overline\sfH_{a}^*(M)
$$
 and since $\varepsilon >0$ is arbitrary, the result follows.

\subsection{Proof of Theorem~\ref{thm:K k-free}} 

Using the inclusion-exclusion principle, we write 
$$
 \sum_{m=1}^M \cK_{m} (a)  \kappa_k(m)  = \sum_{d \le M^{1/k}} \mu(d) 
 \sum_{\substack{m =1\\ m\equiv 0 \pmod {d^k}}}^M  \Kma.
 $$
 We now choose some parameter $D \in [1, M]$ 
and use Corollary~\ref{cor:K-AP} for $d \le D$, while for $d> D$ we use the trivial bound 
$$
 \sum_{\substack{m =1\\ m\equiv 0 \pmod {d^k}}}^M  \Kma \ll M^{1+o(1)} /d^k, 
$$
which is similar~\eqref{eq:Kloos-sums-triv}.   
Hence 
\begin{align*}
 \sum_{m=1}^M \cK_{m} (a)  \kappa_k(m) & \ll \sum_{d \le D} M^{1/2+o(1)} \(Md^{-2k}\)^{\gamma}   +
  \sum_{d > D}  M^{1+o(1)} /d^k\\
  &= M^{1/2+ \gamma +o(1)} \(D^{o(1)}+ D^{1-2k\gamma}\)  +   M^{1+o(1)}D^{1-k} . 
\end{align*}  
We now define $D$ by the equation 
$$
M^{1/2+ \gamma} D^{1-2k\gamma} =   M D^{1-k}
$$
or 
$$
D =  M^{1/(2k)} 
$$
(which in fact does not depend on $\gamma$).  This implies the bound
$$
\sfH_{a,\kappa_k}(M) \ll M^{1/2+ \gamma +o(1)} +  M^{1/2+1/(2k)+o(1)}.
$$
%Since for $k\ge 2$ we have
%$$
%1/(2k) \le 1/4 \le 9/32 \le 1/2- \gamma, 
%$$
%the result now follows.

\subsection{Proof of Theorem~\ref{thm:K phi}} 
Using the well-known elementary formula (see for 
instance~\cite[Theorem~2.3]{Ap})
$$
\varphi(m) 
=  m \sum_{d \mid m} \frac{\mu(d)}{d},
$$
we write 
\begin{align*}
 \sum_{m=1}^M \Kma \varphi(m)  & =  \sum_{m=1}^M m  \Kma   \sum_{d \mid m} \frac{\mu(d)}{d}\\
&  = \sum_{d=1}^M  \frac{\mu(d)}{d} 
 \sum_{\substack{m =1\\ m\equiv 0 \pmod {d}}}^M m \Kma.
\end{align*} 
Hence, by Corollary~\ref{cor:K-AP} 
\begin{align*}
 \sum_{m=1}^M \Kma \varphi(m)   & \ll \sum_{d=1}^M M^{3/2+o(1)} d^{-1}  \(Md^{-2}\)^{\gamma}  \\
  &= M^{3/2+ \gamma +o(1)}\sum_{d=1}^M d^{-1 -2 \gamma}   \le  M^{3/2+ \gamma +o(1)},
\end{align*} 
which concludes the proof. 

\subsection{Proof of Theorem~\ref{thm:K-LinPoly}} 
We use a version of the Weyl differencing and establish the result by induction on $d\ge 0$.

For $d=0$ that is for a pure sum the result is instant from Corollary~\ref{cor:H-Corr-Incomp}.

Now assume that $d \ge 1$.  Given integers 
$h_1, \ldots, h_s$ and a polynomial $g(X) \in \R[X]$ 
we square the sum
$$
S = \sum_{n=1}^N   \prod_{j=1}^s \cK_p(n+h_j)   \ee\(g(n)\).
$$ 
and obtain 
\begin{equation}
\label{eq:S2W}
\begin{split}
S^2 & = \sum_{m,n=1}^N     \prod_{j=1}^s \cK_p(m+h_j)    \prod_{j=1}^s\cK_p(n+h_j)   \ee\(g(n)-g(m)\)\\
&=  2  W   + 
O(N). 
\end{split} 
\end{equation}
where 
$$
W = \sum_{1 \le m < n\le N}   \prod_{j=1}^s \cK_p(m+h_j)    \prod_{j=1}^s\cK_p(n+h_j)   \ee\(g(n)-g(m)\).
$$
Writing $n = m+\ell$ we obtain 
\begin{align*}
W & = \sum_{m=1}^N \sum_{\ell=1}^{N-m}     \prod_{j=1}^s \cK_p(m+h_j)    \prod_{j=1}^s\cK_p(m+\ell+h_j)   \ee\(g(m+\ell)-g(m)\)\\
& =  \sum_{\ell=1}^{N-1}  \sum_{m=1}^{N-\ell}    \prod_{j=1}^s \cK_p(m+h_j)    \prod_{j=1}^s\cK_p(m+\ell+h_j)   \ee\(g(m+\ell)-g(m)\). 
\end{align*} 

We now observe that for every $\ell$ the polynomial $g(X+\ell) - g(X)$ is a polynomial of degree at most $d-1$.

Further more, if   $(h_1, \ldots, h_s) \in \Z^s$ is normal modulo $p$ then for all but at most $s(s-1)/2$ values of $\ell=1, \ldots, N-1$
(avoiding the differences  $|h_i - h_j|$, $1 \le i <j \le s$) the vector
$$
  \(h_1, \ldots, h_s,   h_1+\ell , \ldots, h_s+\ell\) \in \Z^{2s}
$$ 
is also normal modulo $p$. For these exceptional values of $\ell$ we estimate the sums over $m$ trivially as $N$, and apply the
induction assumption for the remaining values of $\ell$.
Hence 
$$
W \ll   N^{2- 2^{-d+1}}p^{2^{-d}} ( \log p)^{2^{-d+1}} + N, 
$$
which after substitution in~\eqref{eq:S2W} implies
$$
S \le   N^{1- 2^{-d}}p^{2^{-d-1}}  ( \log p)^{2^{-d}}  + N^{1/2}. 
$$ 
It remain to note that for $N \le p^{1/2}$ the  result is trivial and for $N \ge p^{1/2}$ we have 
$$
N^{1- 2^{-d}}p^{2^{-d-1}}  ( \log p)^{2^{-d}} \gg N^{1- 2^{-d}}p^{2^{-d-1}} \ge N^{1/2}, 
$$
from which the result follows.

\subsection{Proof of Theorem~\ref{thm:VertCorr-Incomp-II}} We start by noticing that 
$$\big|\cK_p(n)\big|\le 2,  \qquad \forall n \in \N.$$
We proceed also as in the proof Theorem~\ref{conj:K-S} by writing
$$\sum_{n=1}^{N}\cK_p(n) \xi_{n}=\sum_{n=1}^{N} \cK_{p}(n+h) \xi_{n+h}
+O(h).$$
We again fix some  $\varepsilon>0$ and let $\delta$ be as in~\eqref{eq:UC}. 
We choose also the parameters $H$ and $t$ as in~\eqref{eq:Bound t}. Therefore,
$$\sum_{n=1}^{N}\cK_p(n) \xi_{n}=\frac{1}{H}\sum_{h=1}^{H}\sum_{n=1}^{N} \cK_{p}(n+h) \xi_{n+h}
+O(H),$$
Recalling that $\xi_n=f\(T^nx\)$, as in the proof of Theorem~\ref{conj:K-S}, we have
$$\frac{1}{H}\sum_{h=1}^{H}\cK_{p}(n+h) f\(T^{n+h}x\)=
\frac{1}{H}\sum_{h=1}^{H}\cK_{p}(n+h) f\(T^{h}x_{i_n}\)+O(\varepsilon).$$

Hence, it is now enough to show for some $H=o(N)$ that
$$
W= \sum_{n=1}^{N}\frac{1}{H}\sum_{h=1}^{H}\cK_{p}(n+h) f\(T^{h}x_{i_n}\)
$$ 
satisfies
\begin{equation}
\label{eq:W-bound}
W \ll \varepsilon N.
\end{equation}

Now, for fix $s \geq 2$, we apply H\"{o}lder inequality and we relax the dependence of $i_n$ to obtain 
\begin{align*}
|W|^{2s}  &\le  \left |\sum_{n=1}^{N}\frac{1}{H}\sum_{h=1}^{H}\cK_{p}(n+h) f\(T^{h}x_{i_n}\)\right|^{2s} \\
	&\le H^{-2s}N^{2s-1} \sum_{i=1}^t  \sum_{h_1, \ldots h_{2s} =1}^H
	\prod_{j=1}^s f\(T^{h_j} x_{i}\)\overline{ f\(T^{h_{s+j}} x_{i}\)} \\
&\quad \quad \quad\quad	\sum_{n=1}^{N} \prod_{j=1}^s \cK_{p}(m+h_j) \cK_{p}(m+h_{s+j})  \\
& \le t F^{2s} H^{-2s}N^{2s-1}  \sum_{h_1, \ldots h_{2s} =1}^H\left |
\sum_{n=1}^{N} \prod_{j=1}^s \cK_{p}(m+h_j)  \cK_{p}(m+h_{s+j}) \right|. 
\end{align*}
Each tuple $\(h_1, \ldots h_{2s}\) \in [1, H]^{2s}$ that does not satisfy the conditions of Corollary~\ref{cor:H-Corr-Incomp} can be split into $s$ pairs $(j,k)$ such that $h_j \equiv h_k \pmod p$. We thus see that their cardinality may be bounded by
$$
\binom{2s}{s}  s! H^s \le (sH)^s.
$$
We estimate these tuples trivially and bound the others using Corollary~\ref{cor:H-Corr-Incomp}. We thus find

%We see that the number of  choices $\(h_1, \ldots h_{2s}\) \in [1, H]^{2s}$  such that Corollary~\ref{cor:H-Corr-Incomp} does not apply to the inner sum in the above, and thus $\{h_1, \ldots h_{2s}\}$  can be split into $s$ pairs such that $h_j \equiv h_k \pmod p$, for $j \ne k$,   can be estimated as 
%$$
%\binom{2s}{s}  s! H^s \le (sH)^s. 
%$$
%Therefore,  by Corollary~\ref{cor:H-Corr-Incomp}  we have
\begin{align*}
|W|^{2s} &\ll t F^{2s} H^{-2s}N^{2s-1}\((sH)^s N 4^s+H^{2s}N^{1/2}p^{1/4} \)\\
&\ll t F^{2s}\(\frac{s}{H}\)^{s}N^{2s}+t F^{2s}N^{2s-1/2}p^{1/4}\\
&\ll t F^{2s}\(\frac{s}{H}\)^{s}N^{2s}+t F^{2s}N^{2s}\psi(p)^{-1/2}\\
&\ll t F^{2s}  \(\frac{s}{H}\)^{s}N^{2s}, 
\end{align*}
provided that $p$ is sufficiently large in terms of $H$ and $s$. This gives
$$
W \ll  t^{1/2s}  \sqrt{\frac{s}{H}}N.
$$
Choosing $s= \fl{\varepsilon^2H}$ with $\varepsilon<1/2$ and assuming $H$ is large enough so that $s \geq 2$. For the ease of the reader, we repeat the dependencies of the involved parameters
$$
\varepsilon \leftarrow \delta \leftarrow H \leftarrow s \leftarrow p,
$$
where each parameter may depend on all of the precessing ones. In conclusion, by~\eqref{eq:Bound t} we see that $t^{1/2s} \le \sqrt{2}$ and 
 $\sqrt{s/H} \le 2 \varepsilon$. Therefore, we obtain~\eqref{eq:W-bound}
 and complete the proof.

\section*{Acknowledgement}

This work started  during a very enjoyable visit by I.S.
to the Universit{\'e} de Rouen Normandie, whose hospitality is very much appreciated. 
During this work I.S. was also  supported  by ARC Grants~DP170100786 and~DP200100355.  

R.S. would like to extend his gratitude to his employer, the Institute for Mathematical Research (FIM) at ETH Zürich.

\end{document}